\documentclass[english, letterpaper,12pt,reqno]{amsart}

\pdfoutput=1

\usepackage[utf8]{inputenc}
\usepackage[T1]{fontenc}
\usepackage{lmodern}
\usepackage{amsthm, amssymb, amsmath, amsfonts, mathrsfs}
\usepackage[dvipsnames,svgnames,x11names]{xcolor}
\definecolor{mPurple}{HTML}{8959a8}
\usepackage[defaultcolor=purple]{changes}
\usepackage{microtype}

\usepackage[pagebackref,colorlinks=true,pdfpagemode=none,urlcolor=blue,
linkcolor=darkblue,citecolor=darkblue]{hyperref}

\usepackage{amsmath,amsfonts,amssymb,amsthm}
\usepackage{amsthm, amssymb, amsmath, amsfonts, mathrsfs}

\usepackage{mathrsfs}
\usepackage{MnSymbol}
\usepackage{scalerel} 

\usepackage{color}
\usepackage{accents}

\usepackage{bbm}

\definecolor{darkred}{rgb}{0.9,0.1,0.1}
\definecolor{darkgreen}{rgb}{0,0.5,0}
\definecolor{darkblue}{rgb}{0, 0.23, 0.46}

\usepackage{amsmath}
\usepackage{amssymb}
\usepackage{amsfonts}
\usepackage{mathrsfs}
\usepackage{mathtools}

\usepackage[shortlabels]{enumitem}
\usepackage{amscd} 

\usepackage[margin=1in,footskip=.25in]{geometry}

\newtheorem{theorem}{Theorem}[section]
\newtheorem{lemma}[theorem]{Lemma}

\newtheorem{corollary}[theorem]{Corollary}
\newtheorem{proposition}[theorem]{Proposition}

\theoremstyle{remark}
\newtheorem{remark}[theorem]{Remark}

\renewenvironment{proof}[1][Proof]{ {\itshape \noindent {#1.}} }{$\Box$
\medskip}

\numberwithin{equation}{section}
\numberwithin{figure}{section}
\theoremstyle{plain}

\newcommand{\R}{\mathbb{R}}

\newcommand{\1}{\mathbf{1}}
\newcommand{\eps}{\varepsilon}

\newcommand{\dd}{\mathrm{d}}

\newcommand{\Prob}{\mathbb{P}}
\newcommand{\E}{\mathbb{E}}

 \usepackage{bm}

\def\bbE{\mathbb E}
\def\bbVar{\mathbb Var}

\def\bbR{\mathbb R}

\def\rmCov{\mathrm {Cov}}
\def\rmd{\mathrm d}

\def\rmE{\mathrm E}

\def\e{\varepsilon}

\newcommand{\iu}{{\mathbf{i}\mkern1mu}} 

\begin{document}
\title{Time-dependent  averages of a critical long-range stochastic heat equation}
\author[Sefika Kuzgun]{Sefika Kuzgun}

\address[Sefika Kuzgun]{Max Planck Institute for Mathematics in the Sciences,\newline
Inselstraße 22, 04103 Leipzig, Germany.}
\email{sefika.kuzgun@mis.mpg.de}

\author[Ran Tao]{Ran Tao}
\address[Ran Tao]{Department of Mathematics, University of Maryland,\newline 4176 Campus Drive, College Park, MD, USA, 20742.}
\email{rantao16@umd.edu}

\begin{abstract}
We study the time-dependent spatial averages of a critical stochastic partial differential equation, namely the stochastic heat equation in dimension $d\geq 3$ with noise white in time and colored in space with covariance kernel $\|\cdot\|^{-2}$.
    The  solution to this SPDE is a singular measure and was constructed by Mueller and Tribe in \cite{mueller2004singular}. We show that the time-dependent spatial averages of this SPDE over a ball of radius $R$ at time $t$ have different limits under different space-time scales. In particular, when $t\ll R^2$, the central limit theorem holds; when $t= R^2$, the spatial average is a non-Gaussian random variable; when $t\gg R^2$, the spatial average becomes extinct.
    \bigskip


\end{abstract}
\maketitle


\section{Introduction}

\subsection{Main results}
We consider the following stochastic heat equation (SHE) in $\R^d$ for $d\geq 3$ starting from flat initial data \begin{equation}\label{eq:microshe}
\partial_t u(t,x)=\frac{1}{2}\Delta u(t,x) + \kappa u(t,x)\dot{F}(t,x), \quad u(0,x)=1,
\end{equation}
with $\kappa>0$ and $\dot{F}$ being a Gaussian noise that is white in time and has a spatially homogeneous covariance function described by Riesz kernel with index $2$. Formally,  with $\|\cdot\|$ denoting the Euclidean norm in $\R^d$, the covariance function of the noise $\dot{F}$ is expressed as
\[
\E \dot{F}(s,y)\dot{F}(t,x) = \delta(t-s)\|x-y\|^{-2}.
\]

The stochastic heat equation \eqref{eq:microshe} is a {borderline} critical stochastic PDE. Its solution is previously studied and given a meaning in \cite{mueller2004singular} for $0<\kappa < \frac{d-2}{2}$. 
In particular, \cite{mueller2004singular} defines the {(martingale problem) solution} to \eqref{eq:microshe}, and shows that when $0<\kappa < \frac{d-2}{2}$, such solutions  exist ({through Wiener chaos expansion}) and are unique in law. We defer the precise definitions of the solution to Section~\ref{se:solution} below. Moreover, \cite{mueller2004singular} shows that any  solution to \eqref{eq:microshe} is a nonnegative {singular} Radon measure valued process, which we shall denote by $\{u_t(\dd x):t\geq 0\}$. 

For any $R>0$, we use $B_R = \{x \in \R^d:\|x\|< R\}$ to denote the open ball with radius $R$ centered at zero in $\R^d$. The volume of the ball $B_R$ is given by $\omega_d R^d$, with $\omega_d$ being the volume of $B_1$.

Our main result is the following  three different limits of the spatial averages of $u_t(\dd x)$ under different space-time scales. We write $t_R$ to emphasize that the time scale could depend on the radius. Our results cover the cases of {(i)}  when $t$ is fixed and $R$ goes to infinity, or {(ii)}  when $t$ depends on $R$ and  $t,R$ both go to infinity. 

\begin{theorem}\label{thm:main} For  any  $0<\kappa < \frac{d-2}{2}$, let $\{u_t(\dd x):t\geq 0\}$ be {a} solution to \eqref{eq:microshe} as defined in \cite{mueller2004singular}. For $t=t_R>0$, the time-dependent spatial averages $u_{t_R}(B_R)$ have the following limits:
    \begin{enumerate}[leftmargin=*]
        \item\label{it:clt}  When $t_R = o (R^2)$, there exists a Gaussian random variable $\mathcal{N}(0,\sigma^2)$, with  variance  \begin{equation}\label{eq:1stvar}
        \sigma^2 = \kappa^2\int_{B_1}\int_{B_1} \frac{1}{\|x-y\|^2}\dd x \dd y = { \frac{\kappa^2 2^{\,d-1}\,\pi^{\,d-\frac12}\,\Gamma\!\left(\frac{d-1}{2}\right)}
{(d-2)\,\Gamma\!\left(\frac d2\right)\,\Gamma(d)}},
        \end{equation}  such that \begin{equation}\label{eq:clt}
\frac{1}{\sqrt{t_R}R^{d-1}} [u_{t_R}(B_R) -\omega_d R^d]\to \mathcal{N}(0,\sigma^2) \quad \text{in law},
    \end{equation}
    as $R \to \infty$. 
    \medskip
    \item\label{it:scaling}  When $t_R= cR^2$ for some $c>0$, we have \begin{equation}\label{eq:scaling}
\frac{1}{R^d} u_{t_R}(B_R) =u_c(B_1)  \quad \text{in law},
     \end{equation}
    for any $R>0$. Here $u_c(B_1)$ is a  non-degenerate random variable and is positive almost surely. 
        \medskip

\item \label{it:extinct} When $ t_R^{-1} = o(R^{-2})$, we have
\begin{equation}\label{eq:extinct}
\frac{1}{R^d} u_{t_R}(B_R) \to 0  \quad \text{in probability},
    \end{equation}
    as $R\to \infty$. 
    \end{enumerate}
\end{theorem}
We note that the construction of the solution via the chaos expansion \eqref{eq:chaosexpan} below and the scaling invariance of equation \eqref{eq:microshe} are the main tools to study the  spatial averages. 

We first briefly explain what we see from the results. 

For part~\eqref{it:clt} under the condition  $t_R=o(R^2)$, the first chaos is dominant when we take the limit $R\to \infty$, {and all higher order chaoses vanish under rescaling in \eqref{eq:clt}. 
In general, a central limit result does not mean that only the first Wiener chaos contributes to the limit.
 In fact,  there are examples of chaotic central limit theorems
in which the Gaussian limit of the spatial average of an SPDE arises  as a  sum of nontrivial components from all different chaoses (see \cite{nxzquanti}). Our setting, however, turns out to be more similar to the subcritical SPDEs with Riesz-kernel spatial covariance studied in \cite{nualartZheng} (see also Remark~\ref{rm:comapreWNZ} below): the central limit theorem is non-chaotic, in the sense that the limiting Gaussian random variable arises entirely from the first Wiener chaos.


In part~\eqref{it:scaling} when the time variable grows fast enough to be at the borderline $t_R \sim R^2$, the result shows that all chaos now always contributes, and the  random variable $\frac{1}{R^d}u_{t_R}(B_R)$ remains non-Gaussian even when $R\to \infty$. This shows that for this critical SPDE \eqref{eq:microshe}, when the averaging scale matches the diffusion scale,  higher-order interactions accumulate at the same order. 

In part~\eqref{it:extinct}, we have $t_R \gg R^2$, which means that the time scale grows faster than the square of the radius of the averaging region. In this regime, the extinction dominates.

Before proceeding, let us also briefly explain the scaling invariance  of equation \eqref{eq:microshe}. As we will later show in Lemma~\ref{le:scalinv}, for any $R>0, t\geq 0, \eps>0$, we have 
\[u_{t}(B_R)=\eps^{-d/2} u_{\eps t}(B_{\sqrt{\eps}R}) \quad \text{in law}.\]
This scaling invariance comes from the scaling invariance of the noise (see \eqref{eq:noisescalinginv}), and is central to the entire paper. In particular, the prefactor $\eps^{-d/2}$ comes from the  preservation of    total mass under the spatial rescaling $x\to \sqrt{\eps}x$.

Using this scaling invariance relation of \eqref{eq:microshe}, Theorem~\ref{thm:main}~part~\eqref{it:clt} is  equivalent to the following corollary.
\begin{corollary}\label{co:clttozero}
    For any $0<\kappa<\frac{d-2}{2}$,
    when $t_R=o(R^2)$, \begin{equation*}
\frac{1}{\sqrt{t_R/R^2}} [u_{t_R/R^2}(B_1) -\omega_d]\to \mathcal{N}(0,\sigma^2) \quad \text{in law},
    \end{equation*}
    as $R \to \infty$, with $\sigma^2$ as defined in \eqref{eq:1stvar}. 
\end{corollary}
Therefore, Theorem~\ref{thm:main}~part~\eqref{it:clt} can also be understood as the central limit theorem of the random variables $u_t(B_1)$ as $t \to 0$.

    Another case that falls into the category of the extinction (Theorem~\ref{thm:main}~part~\eqref{it:extinct}) is when $R$ is fixed and $t$ goes to infinity. In fact, for any fixed radius $R>0$, 
    \[
    \frac{1}{R^d}u_t(B_R) \to 0 \quad \text{in probability,}\quad \text{as } t\to \infty.
    \]
    This result has already been mentioned in \cite[Theorem 2]{mueller2004singular}, and our Theorem~\ref{thm:main}~part~\eqref{it:extinct} is derived from this result together with the scaling invariance of equation \eqref{eq:microshe}.

\begin{remark}
    The solution theory of the stochastic heat equation \eqref{eq:microshe} when $\kappa\geq \frac{d-2}{2}$ remains open. When $\kappa>\frac{d-2}{2}$, the chaos series used to construct the solution to \eqref{eq:microshe} in \cite{mueller2004singular} (see \eqref{eq:chaosexpan} below) does not converge in $L^2$, putting spatial averages beyond the scope of this discussion.
    
    For the boundary case $\kappa=\frac{d-2}{2}$, a solution constructed via chaos expansion \eqref{eq:chaosexpan} does exists, though \cite{mueller2004singular} leaves open the possibility that  there are
solutions with a different law in this case. 
    One may think of \eqref{eq:microshe} with $\kappa=\frac{d-2}{2}$ as a critical stochastic heat equation with critical scaling. 
    {We believe that} Theorem~\ref{thm:main}~parts~\eqref{it:scaling} and \eqref{it:extinct} still hold for the boundary-case solution constructed via chaos expansion (see \cite[Remark 3 on page 105]{mueller2004singular}). However, our proof of Theorem~\ref{thm:main}~part~\eqref{it:clt} in Section~\ref{se:clt} relies on the condition  $0<\kappa<\frac{d-2}{2}$, and it remains open
    whether the central limit theorem \eqref{eq:clt} extends to cover solutions via chaos expansion in this boundary case $\kappa=\frac{d-2}{2}$.
\end{remark}
\begin{remark}
Another intriguing problem is to investigate the possibility of proving a quantitative central limit theorem under the same conditions as in Theorem~\ref{thm:main} part~\eqref{it:clt}. Previous works, such as \cite{nxzquanti, davarCLT, kim2022limit}, have established various quantitative central limit theorems for spatial averages of stochastic heat equations in one dimension, primarily using Malliavin calculus. Given that the solution to \eqref{eq:microshe} is not defined point-wise, it is of interest to determine whether this approach remains applicable where the Malliavin derivatives cannot be defined point-wise.
\end{remark}

We now introduce the context of our work from the following two perspectives.
\subsection{Spatial averages of SPDE}

The stochastic heat equation \eqref{eq:microshe} is a critical stochastic partial differential equation (SPDE). The noise $F$ has the  scaling invariance 
\begin{equation}\label{eq:noisescalinginv}
\dot{F}(t,x)={ {\eps^2}\dot{F}({\eps^2}{t},{\eps}{x})} \quad \text{in law}
\end{equation}
as a field in $(t,x)\in \R_{\geq 0}\times \R^d$
for any $\eps>0$, which is indicative that the solution will be a borderline distribution. {To see this scaling invariance, we note that for any $\phi, \psi$ in $C_c^{\infty}\left(\mathbb{R}_{\geq 0}\times \mathbb{R}^d\right)$,
\begin{align}\label{noisecovariance}
\rmCov\left[(\dot F,\phi),(\dot F,\psi)\right]=\int_{\bbR^{2d+1}} \phi(t,x)\psi(t,y) |x-y|^{-2} dxdydt .   
\end{align}
 Consider the linear map corresponding to parabolic scaling $\mathcal{E}(t,x)=(\e^2 t,\e x)$ and define $\dot F_\e=\dot F (\mathcal{E}\cdot)$. Also for $\phi,\psi$, let 
 $$\phi^{\e}(\cdot)=|\det \mathcal{E}|^{-1} \phi(\mathcal{E}^{-1}\cdot)=\e^{-(2+d)} \phi(\mathcal{E}^{-1}\cdot). $$
Then, we claim $\e^2\dot F_\e\overset{d}{=}  \dot F $. This can be verified by computing the covariance. By definition of the scaling for the distribution, and then using the covariance structure we have:\begin{align*}
  \rmCov\left[(\dot F_\e,\phi),(\dot F_\e,\psi)\right]= \rmCov\left[(\dot F,\phi^\e),(\dot F,\psi^\e)\right]=\int_{\bbR^{2d+1}} \phi^\e(t,x)\psi^\e(t,y) |x-y|^{-2} dxdydt.   \end{align*}
Then  by a change of variable, we get
\begin{align*}
  \rmCov\left[(\dot F_\e,\phi),(\dot F_\e,\psi)\right]&=\e^{-2(2+d)}\int_{\bbR^{2d+1}} \phi(t/\e^2, x/\e)\psi(t/\e^2, y/\e) |x-y|^{-2} dxdydt   \\&=
  \e^{-2(2+d)}\int_{\bbR^{2d+1}} \phi(s,u)\psi(s,v) | u- v|^{-2} \e^{2d }dudvds  \\
    &= \e^{-4} \rmCov\left[(\dot F,\phi),(\dot F,\psi)\right].
\end{align*}}

Due to the criticality of this SPDE, the solution theory falls outside the well-posedness theories of regularity structures \cite{hairer2014theory, otto2021priori}, paracontrolled distributions \cite{gubinelli2015paracontrolled}, and more recently flow equation \cite{duch2021flow}. Similar critical scaling occurs for the stochastic heat equation with space-time white noise in the critical dimension two, although only the scaling limit of the  
noise-mollified SHE (with a tuning parameter that slowly vanishes) is established in this set-up \cite{csz1, csz2}. 

{Our result considers the spatial average for a critical stochastic PDE with a measure-valued solution. We would like to mention that the critical $2d$ stochastic heat flow (SHF) \cite{csz2} is also measure-valued, but the asymptotic behaviors of its spatial averages remain open.}

Study of the spatial averages of stochastic PDEs started with the work \cite{huang2020central} in the context of stochastic heat equation in one dimension with a space-time white noise potential. Since this work there has been an extensive study of spatial averages of solutions in various different set-ups, see for example, \cite{nualartZheng,huang2020gaussian, nxzquanti} study the SHE with colored noise, \cite{nualart2020spatial} studies the SHE with fractional in time noise, \cite{10.1214/20-AIHP1069, bolanos2021averaging} study the spatial average of stochastic wave equation. Most of these work consider the spatial averages for fixed time, only exception are the works \cite{kim2022limit,davarCLT} where they study time-dependent averages.

Among all the above results, the limiting distribution of the spatial averages, if obtained, is always Gaussian. However, for our critical SPDE \eqref{eq:microshe}, the scaling invariance relation \eqref{eq:noisescalinginv} enables us to obtain a non-Gaussian spatial average \eqref{eq:scaling} under a specific (critical) space-time scale. This is the first non-Gaussian spatial average observed from an SPDE with Gaussian noise  to the best of our knowledge. Previously, it was only shown in \cite{waverosen, heatRosen} that for SPDEs with Rosenblatt noise, one can obtain non-Gaussian spatial averages.

Our result also shows that this transition of Gaussian \eqref{eq:clt} to non-Gaussian \eqref{eq:scaling} spatial averages  takes place when all the chaos begin to contribute in the spatial average limit.

The problem of obtaining non-Gaussian spatial average limit is of its own interest. It has long been conjectured, though not yet proven, that  non-Gaussian stable  distributions could be the limit of spatial averages of the $1d$ stochastic heat equation with space-time white noise when the average is taken over boxes with a proper time-dependent growth rate. 
A similar result was obtained for some parabolic Anderson model on lattice in \cite{ramirez}, though the  Gaussian spatial noise case also remains open.

As a final remark, we point out that in \cite[Theorem 1.7]{nualartZheng}, the authors proved a central limit theorem for the spatial averages of a stochastic heat equation with a spatial covariance function given by the Riesz kernel $\|\cdot\|^{-{p}}$ with index $0<{p}<2$ {in $d\geq 3$}. The SHE in \cite{nualartZheng} is not a critical SPDE and its point-wise solution is well-defined. While our central limit theorem in Theorem~\ref{thm:main}~part~\eqref{it:clt} (for $t$ being fixed) may look like a result of taking $p\to 2^-$ in \cite[Theorem 1.7]{nualartZheng}, it is \emph{not} a simple adaption of \cite{nualartZheng} as the SHE becomes critical when $p\to 2^-$. Our proof of Theorem~\ref{thm:main}~part~\eqref{it:clt} uses new estimates. In particular, our central limit theorem for $p=2$ only holds when the coupling constant is under condition $0<\kappa<\frac{d-2}{2}$, while the result from \cite{nualartZheng} would hold for any positive coupling constant before the noise.
We will discuss this difference further in Remark~\ref{rm:comapreWNZ}.

\subsection{Scaling limit  of long-range SHE}
One may also consider \eqref{eq:microshe} as the macroscopic scaling limit of  a long-range SHE. One closely related work in this line is \cite{gerolla2023fluctuations}, 
where the authors study the scaling limit of the long-range nonlinear stochastic heat equation in dimension $d\geq 3$. If we only consider \cite{gerolla2023fluctuations} for the linear case,  we may interpret its result and compare it with ours in the following sense.

Consider 
the stochastic heat equation  in $\R^d$ for $d\geq 3$ starting from flat initial data \begin{equation}
\partial_t v(t,x)=\frac{1}{2}\Delta v(t,x) + \kappa v(t,x)\dot{V}(t,x), \quad v(0,x)=1,
\end{equation}
with $\kappa>0$ and $\dot{V}$ being  a space-time mean zero Gaussian noise with formal covariance function
\[
\E \dot{V}(s,y)\dot{V}(t,x) = \delta(t-s)\left(\|x-y\|^{-p}\wedge 1\right),
\]
for some $p \in [2,d)$.

By \cite{mueller2004singular},
when $p=2$, $\kappa<\frac{d-2}{2}$, and $t_R=cR^2$ for some $c>0$, we have
\begin{equation}\label{eq:ournonGaussian}
\frac{1}{R^{d}}\int_{B_R}  v(t_R,x)  \dd x \to u_c{(B_1)}  \quad \text{in law,}
\end{equation}
as $R \to \infty$. The limit  is the same random variable as in \eqref{eq:scaling} and is  non-Gaussian. Given our result Theorem~\ref{thm:main}~part~\eqref{it:scaling}, another heuristic to understand this limit is that the spatial average of $v$ would approximate the spatial average of $u$ in large scales.

On the other hand, by  \cite[Theorem 1.3]{gerolla2023fluctuations} with $\sigma(x)=x$,  when $p \in (2,d)$, $\kappa$ is small enough, and $t_R=cR^2$ for some $c>0$, we have that for any smooth test function with compact support $g:\R^d\to \R$, 
\begin{equation}\label{eq:ghlgaussian}
\frac{1}{R^{d+1-\frac{p}{2}}}\int_{\R^d} \left[v(t_R,x)- {1}\right]g\left(x/R\right) \dd x \to \int_{\R^d} \mathcal{U}(c, x) g(x)\dd x \quad \text{in law},  
\end{equation}
as $R \to \infty$,
where $\mathcal{U}$ is the solution to the additive stochastic heat equation (Edwards-Wilkinson equation) 
$$
\partial_t \mathcal{U}(t, x)=\frac{1}{2} \Delta \mathcal{U}(t, x)+\kappa \eta(t, x), \quad \mathcal{U}(0, \cdot) \equiv 0,
$$
with noise $\eta$ being  a space-time mean zero Gaussian noise with formal covariance function
\[
\E \eta(s,y)\eta(t,x) = \delta(t-s)\|x-y\|^{-p}.
\]
In particular, the limit in \eqref{eq:ghlgaussian} is Gaussian.

This comparison shows that the case $p=2$ distinguishes itself from the cases with $2<p<d$ as its large-scale fluctuations remain multiplicative, thus  non-Gaussian, under diffusive  scaling. When $p=2$, all the chaos contribute to the fluctuation at the space-time scale $t=cR^2$, while when $2<p<d$, only the first chaos from the Riesz-kernel-correlated noise contributes after normalization.

Another related literature is \cite{gu2020fluctuations}, where the authors consider the scaling limit of mollified SHE in $d\geq 3$ with compactly supported covariance function. For this model, the fluctuations of the diffusively scaled solution converge to an Edwards-Wilkinson limit with space-time white noise and with a nontrivial effective variance. See also \cite{cosco2, cosco}.

\section{Preliminaries}
\subsection{Solution theory}\label{se:solution}
As discussed above, for the critical stochastic heat equation \eqref{eq:microshe}, the first challenge is to make sense of its solution.
In this subsection, we  introduce the definition of  the
 solution to \eqref{eq:microshe} and explain one specific construction of the solution from the Wiener chaos expansion. All of these results follow from \cite{mueller2004singular}.

Let $(\Omega, \mathcal{F}, \{\mathcal{F}\}_t,\Prob )$ be a filtered probability space. Let  $\mathcal{M}$ denote the nonnegative Radon measures on $\R^d$, $\mathcal{C}_c$ denote
the space of continuous functions
on $\R^d$ with compact support,  and $\mathcal{C}^k_c$ denote
the space of functions in $\mathcal{C}_c$ with $k$ continuous derivatives. We define $\mu(f):= \int_{\R^d} f(x) \mu(d x)$ for any $\mu \in \mathcal{M}$ and $f$ being integrable on $\R^d$. We also equip the space $\mathcal{M}$ with the vague topology. We denote the standard $d$-dimensional heat kernel as $G_t(x):=(2 \pi t)^{-d / 2} \exp \left(-|x|^2 / 2 t\right)$ for any $x \in \R^d$ and $t>0$.

In
\cite{mueller2004singular}, the {(martingale problem) solution} to \eqref{eq:microshe} is defined as any adapted continuous $\mathcal{M}$ valued process 
$\{u_t(\dd x):t\geq 0\}$ such that it satisfies the following three conditions:
\begin{enumerate}[leftmargin=*]
    \item $\Prob(u_0(\dd x)=\dd x)=1$;
        \medskip
    \item $\{u_t(\dd x):t\geq 0\}$ satisfies
\begin{equation*}\label{eq:l1mean}
     \E[u_t(f)]=\int_{\R^{d}} f(x) \dd x,
\end{equation*}
and 
\begin{equation}\label{eq:l2bound}
\mathbb{E} \left[
\left(u_t(f)\right)^2\right]\leq  C \int_{\R^{4d}} G_t\left(x-x'\right) G_t\left(y-y'\right) f\left(x'\right) f\left(y'\right)\left(1+\frac{t^\alpha}{|x-y|^\alpha\left|x'-y'\right|^\alpha}\right) \dd x \dd y \dd x'\dd y',
\end{equation}
for some positive constant $C=C(d,\kappa)$  and $\alpha=\frac{d-2}{2}-[(\frac{d-2}{2})^2-\kappa^2]^{1/2}$;
    \medskip
\item $\{u_t(\dd x):t\geq 0\}$ satisfies the martingale problem, i.e.,\ for all $f \in \mathcal{C}_c^2$, \begin{equation}\label{eq:martingale}
z_t(f)=u_t(f)-u_0(f)-\int_0^t \frac{1}{2} u_s(\Delta f) \dd s
\end{equation}
is a continuous local $\mathcal{F}_t$ -martingale with finite quadratic variation given by
\begin{equation}\label{eq:martingaleqv}
\begin{aligned}
\langle z(f)\rangle_t=\kappa^2 \int_0^t \int_{\R^{2d}} \frac{f(x) f(y)}{|x-y|^2} u_s(\dd y) u_s(\dd x) \dd s.
\end{aligned}
\end{equation}
\end{enumerate}

For any $0<\kappa < \frac{d-2}{2}$, \cite{mueller2004singular} shows that such  solution to \eqref{eq:microshe} exists and is unique in law. 
In particular, one construction of the  solution to \eqref{eq:microshe} is from the Wiener chaos expansion. We use this construction throughout our paper, so we explain it in detail.

For any $f \in \mathcal{C}_c$, define
$I^{(n)}_t(f)$ as the integral of $f$ over the $n$-th order Wiener chaos, given by
\[
I^{(n)}_t(f):=\int_{\bbR^{d}} f(x)    I^{(n)}(t,x)\rmd x,
\]
with $I^{(n)}(t,x)$ being defined  as $I^{(0)}(t,x) =1 $, and for any $n\geq 1$, setting $s_{n+1} = t$, $y_{n+1}=x$,  \[
I^{(n)}(t,x) :=\kappa^n \int_0^{s_{n+1}} \int_0^{s_n} \ldots \int_0^{s_2} \int_{\R^{n d}} \prod_{i=1}^n G_{s_{i+1}-s_i}\left(y_{i+1}-y_i\right) {F}\left(\dd s_i, \dd y_i \right),
\]
where the stochastic integrals with ${F}$ are well-defined as It\^o-Walsh integrals  \cite{walsh} 
 and can also be interpreted  as multiple Wiener-It\^o integrals {with respect to ${F}$} (see \cite[Section 1.1.2]{nualart2006malliavin}).
Note that one can easily extend the above definition of $I^{(n)}_t(f)$ to $f$ being indicator functions $\1_{B_R}$ for any $R>0$, 
{as the indicator functions convoluted with heat kernel are smooth.}

For any $0<\kappa < \frac{d-2}{2}$, by \cite{mueller2004singular}, there exists an adapted $\mathcal{M}$ valued process $\{\tilde{u}_t(\dd x):t\geq 0\}$ such that for any $t\geq 0$,
\begin{equation}\label{eq:chaosexpan}
\tilde{u}_t(f)=\sum_{n=0}^{\infty} I_t^{(n)}(f), \quad \text{for all } f \in \mathcal{C}_c,  \quad \Prob\text{-almost surely}.
\end{equation}
In fact, it is easy to check that for any $f\in \mathcal{C}_c$,
$\E\left[\sum_{n=0}^{\infty} I_t^{(n)}(f)\right]=\int_{\R^{d}} f(x) \dd x$,
and by \cite[Lemma 3 \& Lemma 5]{mueller2004singular}, when $0<\kappa < \frac{d-2}{2}$,
$\E\left[\left(\sum_{n=0}^{\infty} I_t^{(n)}(f)\right)^2\right]$ is bounded above by the right-hand-side of \eqref{eq:l2bound}. The second moment bound implies that there is a regularization \cite[Theorem 2.3.3]{ito} so that  for any $t>0$, there exists a random distribution $\tilde{u}_t(\dd x)$ satisfying \eqref{eq:chaosexpan}. 
Moreover, \cite{mueller2004singular} shows that for any $t>0$,  the random distribution satisfying \eqref{eq:chaosexpan} is a nonnegative singular Radon measure, and the $\mathcal{M}$ valued process $\{\tilde{u}_t(\dd x):t\geq 0\}$ has a modification continuous in time {(with respect to the vague topology)}. 
This continuous modification satisfies the martingale problem \eqref{eq:martingale}--\eqref{eq:martingaleqv}.
Therefore, combining all the results above, the continuous modification of the process $\{\tilde{u}_t(\dd x):t\geq 0\}$ with $\tilde{u}_t$  satisfying \eqref{eq:chaosexpan} is a martingale problem solution to \eqref{eq:microshe}.

We also record the following Feynman-Kac type formula for the  chaos expansion, where we refer to \cite[Lemma 5]{mueller2004singular} for details. By Urysohn's lemma and monotone convergence theorem, one can extend  \cite[Lemma 5]{mueller2004singular} to the indicator functions $f=\1_{B_R}$. 

Let
\begin{align}\label{eq:betat}
    \beta_t(x):=\int_0^t \frac{\dd s}{\left\|x+\sqrt{2}W_{s}\right\|^{2}},
\end{align} 
where $(W_t)_{t\geq 0}$ denotes a standard $d$-dimensional Brownian motion starting from $W_0=0$ independent of $\xi$. We let $\rmE$ denote the expectation with respect to this Brownian motion. 
Then for each $n\geq 1$ and any $t>0, R>0$, by It\^o isometry, with $s_{n+1}=t$ and $y_{n+1}=x_1$,$z_{n+1}=x_2$,
\begin{equation}\label{eq:chaosvar}
\begin{aligned}
    \bbE\left[I^{(n)}_t(\1_{B_R})^2\right] &= \kappa^{2n} \int_{B_R}\int_{B_R}\int_0^{s_{n+1}} \int_0^{s_n} \ldots \int_0^{s_2} \int_{\R^{2n d}} \\
    &\prod_{i=1}^n G_{s_{i+1}-s_i}\left(y_{i+1}-y_i\right) G_{s_{i+1}-s_i}\left(z_{i+1}-z_i\right) \frac{1}{\|y_i-z_i\|^2} \dd y_i \dd z_i \dd s_i\dd x_1 \dd x_2 \\
    &=  \frac{\kappa^{2n}}{n!} \int_{B_R}\int_{B_R}\rmE\left[  \prod_{i=1}^n \int_0^t \frac{1}{\left\|x-y+\sqrt{2}W_{s_i}\right\|^{2}} \dd s_i\right]\dd x\dd y
    \\
    &= \frac{\kappa^{2n}}{n!}\int_{B_R}\int_{B_R}\rmE[  \beta_t(x-y)^n]\dd x\dd y,
\end{aligned}
\end{equation}
where the second equation is obtained by  using  the finite dimensional distributions of the
Brownian bridges.

By \cite{mueller2004singular}, for any $0<\kappa<\frac{d-2}{2}$, the series sum
$\sum_{n=0}^{\infty}\bbE\left[I^{(n)}_t(\1_{B_R})^2\right]$
is convergent, and we have
\begin{align*}
\bbE[\tilde{u}_t(B_R)^2]&=\sum_{n=0}^{\infty} \bbE\left[I^{(n)}_t(\1_{B_R})^2\right]\\&=
\sum_{n=0}^{\infty}\frac{\kappa^{2n}}{n!}\int_{B_R}\int_{B_R}\rmE[  \beta_t(x-y)^n]\dd x\dd y=\int_{B_R}\int_{B_R}\rmE\left[ \exp{\left(  \kappa^2\beta_t(x-y)\right)}\right]\dd x\dd y.
\end{align*}

\subsection{Technical lemmas}\label{se:fourier}
Since we will conduct our main analysis on the Fourier space, we introduce some  auxiliary lemmas related to the Fourier transform of the Riesz kernel and the indicator functions. 

Throughout the paper, we take the Fourier transform of any tempered distribution $\mu$ on $\R^d$ in the convention of
\[\mathscr{F}\mu(x)  :=\int_{\R^d} e^{-\iu x \cdot \xi} \mu(\dd \xi).\]
Under this convention, for any functions $u, v \in 
L^2\left(\R^d \right)$, the Parseval formula is
\begin{equation*}\label{eq:parseval}
   \int_{\R^d} u(\xi)\overline{v(\xi)}\dd \xi  =\frac{1}{(2 \pi)^d} \int_{\R^d} \mathscr{F}u(x)\overline{\mathscr{F}v(x)}d x. 
\end{equation*}
When $\mu$ is in the space of  tempered distributions, its Fourier transform $\mathscr{F}\mu$ is well-defined through the Parseval formula.
By Bochner's theorem,  for any continuous positive definite function $\gamma: \mathbb{R}^d \rightarrow \mathbb{R} \cup\{\infty\}$, there exists a nonnegative tempered distribution $\mu$ on $\R^d$ so that $\gamma(x)=$ $\mathscr{F} \mu(x)$ and $\mu$ is known as the spectral measure.

In particular, for the Riesz kernel $\gamma(x)={1}/{\|x\|^2}$ on $\R^d$ with $d \geq 3$, its spectral measure is
\begin{equation}\label{eq:spectralmeasure}
\mu(\dd \xi) =  c_d {\|\xi\|^{-d+2}}\dd \xi, \quad \text{ with coefficient } \quad c_d=  \frac{\Gamma(\frac{d-2}{2})}{\pi^{d/2}2^2}.
\end{equation}
The spectral measure \eqref{eq:spectralmeasure} does not satisfy Dalang's condition \cite{dalang1999extending}:
\[
\int_{\mathbb{R}^d} \frac{\mu(d \xi)}{1+\|\xi\|^2}<\infty.
\]

We now give the composition formula of the Riesz kernels. This formula is well-known and we state it here for the convenience of the readers. We omit the proof, as it follows from an elementary computation with the relation of gamma and beta functions.
\begin{lemma}\label{le:rieszcomp}
    For any $0<\alpha,\beta,\alpha+\beta<d$ and $x\neq y$,
    \begin{align*}
      \int_{\bbR^d} \|x-z\|^{\alpha-d} \|z-y\|^{\beta -d}dz=k_{\alpha,\beta,d} \|x-y\|^{\alpha+\beta-d}, 
  \end{align*}
  with the coefficient \begin{align*}
    k_{\alpha,\beta,d} =\pi^{d/2}\frac{\Gamma(\frac{\alpha}{2})\Gamma(\frac{\beta}{2})\Gamma(\frac{d-\alpha-\beta}{2})}{\Gamma(\frac{d-\alpha}{2})\Gamma(\frac{d-\beta}{2})\Gamma(\frac{\alpha+\beta}{2})}.
  \end{align*}
\end{lemma}

Another useful lemma is the Fourier transform of the indicator function over a ball of radius $R$ in $\R^d$. We refer to \cite[Lemma 2.1]{nualartZheng} for its proof.

\begin{lemma}[{\cite[Lemma 2.1]{nualartZheng}}]\label{lem:el1}
For any $R>0$, we have \[
\mathscr{F}\1_{B_R}(\xi)=\int_{B_R}e^{-\iu x\cdot \xi }\dd x=(2\pi R)^{d/2} \|\xi\|^{-d/2}J_{d/2}(R\|\xi\|), \quad  \text{for }  \xi\in\bbR^d,\]
where $J_{d/2}(\cdot)$ is the Bessel function of the first kind with order $d/2$:
$$
J_{d/2}(x)=\frac{(x / 2)^{d/2}}{\sqrt{\pi} \Gamma\left(\frac{d+1}{2} \right)} \int_0^\pi(\sin \theta)^{d} \cos (x \cos \theta) d \theta, \quad \forall x \in \mathbb{R}.
$$
\end{lemma}
By \cite[page 134]{lebedev1972special}, the Bessel function
$J_{d/2}(\cdot)$ has the following asymptotic behaviors:
\begin{align*}
        J_{d/2}(x) &\sim \sqrt{2/(\pi x)}\cos \left(x-\frac{(d +1)\pi}{4}\right) \quad \text{as  }x\to \infty,\\ \text{and} \quad 
      J_{d/2}(x) &\sim \frac{x^{d/2}}{2^{d/2} \Gamma(d/2+1)} \quad \text{as  }x\to 0.
    \end{align*}
Consequentially, we have
$$
\sup_{x \in \R}  \left|J_{d/2}(|x|)\right| <\infty,
$$
and there exists a constant $C>0$ so that
$$|J_{d/2}(|x|)| \leq C|x|^{-1/2}, \quad \text{for any } x \in \R\backslash\{0\}.$$
For $0<\alpha<d+1$, it is easy to verify that the following integral is finite:
\begin{equation}\label{eq:finiteJint}
    \int_{\bbR^d} \|\eta\|^{-2d+\alpha} J_{d/2}(\|\eta\|)^2d\eta<\infty.
\end{equation}

\section{Proof of Theorem~\ref{thm:main} and Corollary~\ref{co:clttozero}}
\subsection{Scaling}
The following lemma is from the scaling invariance relation of equation \eqref{eq:microshe}. Using it together with \cite[Theorem 2 \& 3]{mueller2004singular} proves Theorem~\ref{thm:main}~part~\eqref{it:scaling} and \eqref{it:extinct}.

\begin{lemma}\label{le:scalinv}
    For any radius $R>0$,
    any time $t\geq 0$, and any $\eps >0$, we have
\begin{equation}\label{eq:equalinlaw}
            u_{t}(B_R)=\eps^{-d/2} u_{\eps t}(B_{\sqrt{\eps}R}) \quad \text{in law}.
    \end{equation}
\end{lemma}
\begin{proof}
{For any fixed $\eps>0, t\geq0$, define the measure $v^{\eps}_t(\dd x)\in \mathcal{M}$  so that for any Borel measurable set $A \subset \R^d$, we have $v^{\eps}_t(A)=\eps^{-d/2} u_{\eps t}(\sqrt{\eps}A)$, where $\sqrt{\eps}A =\{ \sqrt{\eps} x: x\in A\}$. By the scaling invariance relation \eqref{eq:noisescalinginv}, $\{v^{\eps}_t(\dd x): t \geq 0\}$ is also a solution to \eqref{eq:microshe}. The uniqueness in law of the solution to \eqref{eq:microshe} then implies \eqref{eq:equalinlaw}.
The result that$\{v^{\eps}_t(\dd x): t \geq 0\}$ is a solution to \eqref{eq:microshe} is straightforward to check, which we omit here.} One can also   check the invariance chaos by chaos,  using \eqref{eq:noisescalinginv}. A result more general has also been previously mentioned in \cite[Lemma 2]{mueller2004singular}. We note that the prefactor $\eps^{-d/2}$ is chosen here so that the rescaled measure-valued process is the solution to \eqref{eq:microshe} with the same initial data.
\end{proof}

\begin{proof}[Proof of Theorem~\ref{thm:main}~parts~\eqref{it:scaling} and \eqref{it:extinct}]
For any $R>0$ and $t=t_R>0$, applying Lemma~\ref{le:scalinv} with $\eps = R^{-2}$, we have
    \[
    \frac{1}{R^d}u_{t_R}(B_R)=  u_{t_R/R^2}(B_{1}) \quad \text{in law}.
    \]
    
    When $t=t_R=cR^2$ for some $c>0$, 
 we have proved \eqref{eq:scaling}. By \eqref{eq:chaosvar}, $\bbVar[u_c(B_1)]>0$. The positivity {of $u_c(B_1)$} follows from \cite[Theorem 3 ii]{mueller2004singular}.
    
    When $t=t_R$ with $t_R^{-1}=o(R^{-2})$, we have $t_R/R^2\to \infty$ as $R\to \infty$.
    By \cite[Theorem 2 ii]{mueller2004singular}, $u_{t}(B_{1})$ converges to zero in probability as $t\to \infty$. Thus the proof is complete.
\end{proof}

We also note that, assuming 
Theorem~\ref{thm:main}~part~\eqref{it:clt}
holds,
Corollary~\ref{co:clttozero} is another consequence of Lemma~\ref{le:scalinv} with $\eps = R^{-2}$.

\subsection{Central limit theorem}\label{se:clt}
In this subsection, we prove Theorem~\ref{thm:main}~part~\eqref{it:clt}. 

Without loss of generality, we only need to prove Theorem~\ref{thm:main}~part~\eqref{it:clt} for the case of $t_R=1$ being fixed. This is also due to the scaling invariance relation. In fact, for any $t=t_R>0$, applying Lemma~\ref{le:scalinv} with $\eps = {t_R}^{-1}$, we have
\begin{align*}
   &\frac{1}{\sqrt{t_R}R^{d-1}} \left[u_{t_R}(B_R) -\omega_d R^d\right] \stackrel{\text { law }}{=} \frac{1}{\sqrt{t_R}R^{d-1}} \left[t_R^{d/2}u_{1}\left(B_{{t_R^{-1/2}R}}\right) -\omega_d R^d\right] \\
    &= \frac{t_R^{d/2}}{\sqrt{t_R}R^{d-1}} \left[u_{1}\left(B_{t_R^{-1/2}R}\right) -\omega_d \left(t_R^{-1/2}R\right)^d\right] =\frac{1}{\left({t_R}^{-1/2}R\right)^{d-1}} \left[u_{1}\left(B_{t_R^{-1/2}R}\right) -\omega_d \left(t_R^{-1/2}R\right)^d\right]
   \\&= \frac{1}{ \tilde{R}^{d-1}} \left[u_{1}\left(B_{\tilde{R}}\right) -\omega_d  \tilde{R}^d\right],
\end{align*}
where we define $\tilde{R}:=t_R^{-1/2}R$ in the last equation. Note that when $t_R=o(R^2)$, we have $\tilde{R} \to \infty$ as $R \to \infty$, so  it is sufficient to prove \eqref{eq:clt}  for the fixed time case $t_R=1$.

Additionally, by \cite{mueller2004singular}, the (martingale problem) solution to \eqref{eq:microshe} is unique in law, so we only need to prove \eqref{eq:clt} for the process $\{\tilde{u}_t(\dd x):t\geq 0\}$, where  $\{\tilde{u}_t(\dd x):t\geq 0\}$ is the adapted $\mathcal{M}$ valued process defined by the chaos expansion \eqref{eq:chaosexpan}.

From now on, we will only consider the series 
$
\tilde{u}_1(B_R)=\sum_{n=0}^{\infty} I_1^{(n)}(\1_{B_R})$.
It is immediate that 
\[
\bbE[\tilde{u}_1(B_R)]=\bbE\left[I^{(0)}_1(\1_{B_R})\right] = \omega_dR^d,
\]
and
Theorem~\ref{thm:main}~part~\eqref{it:clt} would follow from the following two results.
\begin{proposition}\label{prop:firstchaos}
   For any $\kappa>0$, we have
    \begin{align*}
      \lim_{R\to \infty}  \frac{\bbVar\left[I_1^{(1)}(\1_{B_R})\right]}{R^{2d-2}}  = \sigma^2,
    \end{align*}
    where $\sigma^2$ is given by \eqref{eq:1stvar}.
\end{proposition}

\begin{proposition}\label{prop:higher}
   For  any $0<\kappa<\frac{d-2}{2}$, we have \begin{align*}
       \lim_{R\to \infty} \frac{1}{ R^{2d-2}}\sum_{n=2}^{\infty}  \bbVar\left[I_1^{(n)}(\1_{B_R})\right] =0.
   \end{align*} 
\end{proposition}

\begin{remark}\label{rm:comapreWNZ}
Before proving these two propositions, we point out how our results are similar or not similar to  \cite[Theorem 1.7]{nualartZheng}. In \cite[Theorem 1.7]{nualartZheng}, the authors proved a central limit theorem for the spatial averages of a stochastic heat equation with a spatial covariance function given by the Riesz kernel $\|\cdot\|^{-{p}}$ {with index $0<{p}<(2 \wedge d)$ for $d\geq 1$}, using the same method as checking the order of the variance of the first order chaos and then the sum of all higher order chaos. 

However, although Proposition~\ref{prop:firstchaos} and Lemma~\ref{le:nthchaos} below are indeed similar to the analogous results in the proof of \cite[Theorem 1.7]{nualartZheng}, our proof of Proposition~\ref{prop:higher} is not an adaption of \cite{nualartZheng}.

From the viewpoint of \eqref{eq:microshe} being a critical SPDE,  we only expect to prove Proposition~\ref{prop:higher} under the restriction  $0<\kappa<\frac{d-2}{2}$. This is a strong indication that \cite{nualartZheng} cannot be directly adapted to prove Proposition~\ref{prop:higher} here. In fact, if we consider the counterpart of \eqref{eq:microshe} {(still in $d\geq 3$)} with Riesz index $0<p<2$, the central limit theorem from  \cite{nualartZheng} should hold for  all $\kappa>0$. However, in our case ($p=2$), for any fixed-radius ball $B_R$ with $R>0$, the chaos expansion $\sum_{n=0}^{\infty}I_1^{(n)}(\1_{B_R}) $ is not bounded in $L^2(\Omega)$ when $\kappa>\frac{d-2}{2}$ (see \cite[page 114]{mueller2004singular}). This shows that it would be unreasonable to expect a central limit theorem for all $\kappa>0$, and it further suggests that the estimates from \cite{nualartZheng} would not work   here.

If one dives into the technical detail, when Dalang's condition holds, one can always choose a sufficiently large $N$ so that the integral 
$
\int_{|\xi|\geq N} \frac{1}{\|\xi\|^2}\mu(\dd \xi)
$ becomes arbitrarily small. Here $\mu$ is the spectral measure of the spatial covariance function as discussed in Section~\ref{se:fourier}. The estimates from 
\cite{nualartZheng} rely on this fact, and it is also the reason why the central limit theorem would hold for  any coupling constant $\kappa>0$ before the noise when the Riesz kernel is of index $0<p<2$. 
When the index of the Riesz kernel becomes $p=2$, Dalang's condition no longer holds, and the above integral of the tail is not finite for any arbitrary $N$. The estimates from \cite{nualartZheng} thus become inapplicable, and the estimates we use to prove Proposition~\ref{prop:higher} are new.
\end{remark}

\subsubsection{Notations and auxiliary lemmas}
We first give some notations. For any $n \geq 1$, $1\leq k \leq n-1$, we use $\bm{\eta}_n=(\eta_1,
\dots, \eta_n) \in \bbR^{nd}$ and $\bm{\eta_{k,n}}=(\eta_{k+1},
\dots, \eta_n) \in \R^{(n-k)d}$ to denote the $n$-tuple and $(n-k)$-tuple spatial variables respectively. We also use ${\bm{w}_n}=(w_1,\dots,w_n)\in  \bbR^{n}_{> 0}$ for the $n$-tuple time variables.
Since we only need to consider for $t_R=1$, we set
\begin{align}
\Delta_n&:=\{\bm{s}_n \in  \bbR^{n}_{> 0}: 0<s_n<\dots <s_1<1\}, \label{eq:delta}\\ S_n&:=\{\bm{w}_n\in  \bbR^{n}_{> 0}: w_1+\dots+w_n< 1\}, \label{eq:triangle}
\end{align}
and define \begin{equation}\label{phi}
\varphi(\bm{\eta}_{n}):=\int_{S_{n}} e^{-\sum_{i=1}^n{w_i \|\eta_i\|^2}}\dd{\bm{w}}.
\end{equation}
Now for any $n\geq 1$, we  give  the formula of the variance of $I_1^{(n)}(\1_{B_R})$ on the Fourier side. This computation is similar to the one in \cite[Theorem 1.7]{nualartZheng} when an SHE with a spatial covariance function given by the Riesz kernel $\|\cdot\|^{-p}$ {for $d\geq 1$ with index $0<p<(2\wedge d)$} is considered. For completeness, we give the proof.
\begin{lemma}\label{le:nthchaos}
For any $n\geq 1$, $R>0$, setting  $\eta_0=0$, we have
\begin{equation*}
   \bbVar[I_1^{(n)} (\1_{B_R})]=\kappa^{2n}(2\pi)^{d} R^{2d-2n}
     c_{d}^n \int_{\bbR^{nd}} \prod_{j=1}^n \|\eta_j-\eta_{j-1}\|^{-d+2} \|\eta_n\|^{-d}J_{d/2}(\|\eta_n\|)^2\varphi(\bm{\eta}_n/R)  \dd \bm{\eta}_n,\end{equation*}
     where $c_d$ is defined as in \eqref{eq:spectralmeasure}, and  the functions $\varphi$ and $J_{d/2}$ are  defined as in \eqref{phi} and Lemma~\ref{lem:el1}.
\end{lemma}
\begin{proof}
{Recall that in \eqref{eq:betat}, we defined $\beta_1(x)=\int_0^1 \frac{\dd s}{\left\|x+\sqrt{2}W_{s}\right\|^{2}}$ for any $x \in \R^d$.}
By \eqref{eq:chaosvar}, for any $n\geq 1$,
\begin{equation}\label{eq:varforn}
\bbVar[I_1^{(n)} (\1_{B_R})]=
\bbE\left[I^{(n)}_1(\1_{B_R})^2\right] = \frac{\kappa^{2n}}{n!}\int_{B_R}\int_{B_R}\rmE[  \beta_1(x-y)^n]\dd x\dd y.
\end{equation}
By the Fourier transform of Gaussian random variables and \eqref{eq:spectralmeasure}, we have
\begin{align*}
    \rmE[\beta_1(x-y)^n]&= c_{d}^n \int_{[0,1]^n} \int_{\bbR^{nd}} \prod_{j=1}^n \|\xi_j\|^{-d+2} e^{-\iu (x-y)\cdot (\xi_1+\cdots+ \xi_n)} \rmE[e^{-\iu \sum_{j=1}^n\xi_j\cdot \sqrt{2}W_{s_j}}] \dd \bm{\xi}_n \dd\bm{s}_n  \\&=
    c_{d}^n n! \int_{\Delta_n} \int_{\bbR^{nd}} \prod_{j=1}^n \|\xi_j\|^{-d+2} e^{-\iu (x-y)\cdot (\xi_1+\cdots+ \xi_n)} e^{-\sum_{j=1}^n(s_j-s_{j+1}) \|\xi_1+\dots +\xi_j\|^2  } \dd \bm{\xi}_n \dd\bm{s}_n,
\end{align*}
where we set $s_{n+1}=0$,  $c_d$ is  as defined in \eqref{eq:spectralmeasure},  and $\Delta_n$ is as defined in \eqref{eq:delta}. Now making the change of variables $\eta_j:=\xi_1+\dots+\xi_j$ and $w_j:=s_j-s_{j+1}$, we can rewrite the above   as 
\begin{align*}
   &
    c_{d}^n n! \int_{S_n} \int_{\bbR^{nd}} \prod_{j=1}^n \|\eta_j-\eta_{j-1}\|^{-d+2} e^{-\iu (x-y)\cdot \eta_n} e^{-\sum_{j=1}^nw_j \|\eta_j\|^2  } \dd \bm{\eta}_n \dd\bm{w}_n \\
    &=c_{d}^n n!  \int_{\bbR^{nd}} \prod_{j=1}^n \|\eta_j-\eta_{j-1}\|^{-d+2} e^{-\iu (x-y)\cdot \eta_n} \varphi(\bm{\eta}_n) \dd \bm{\eta}_n,
\end{align*}
where $S_n$  is as defined in \eqref{eq:triangle} and $\varphi$ is as defined in \eqref{phi}.
Now plugging back into \eqref{eq:varforn}, we have \begin{align*}
    \bbVar[I_1^{(n)} (\1_{B_R})] &= \kappa^{2n}c_{d}^n  \int_{\bbR^{nd}} \prod_{j=1}^n \|\eta_j-\eta_{j-1}\|^{-d+2} \left(\int_{B_R}\int_{B_R}e^{-\iu (x-y)\cdot \eta_n}\dd x\dd y\right)\varphi(\bm{\eta}_n) \dd \bm{\eta}_n \\&=\kappa^{2n}(2\pi R)^{d}c_{d}^n\int_{\bbR^{nd}} \prod_{j=1}^n \|\eta_j-\eta_{j-1}\|^{-d+2} \|\eta_n\|^{-d}J_{d/2}(R\|\eta_n\|)^2 \varphi(\bm{\eta}_n) \dd \bm{\eta}_n,
\end{align*}
where  the last line  follows from Lemma~\ref{lem:el1}. The proof is  now complete by applying another change-of-variable $\eta_j \to \eta_j/R$ for all $j=1,2,\dots,n$.
\end{proof}

Another useful lemma is the bounds of $\varphi(\bm{\eta}_{n})$:
\begin{lemma}\label{lem:phibound}
    For any $n\geq 1$ and $\bm{\eta}_n \in \R^{nd}$,
\begin{align}\label{phibound1}
\varphi(\bm{\eta}_n)\leq \frac{1}{n!}.
  \end{align}
  Thus for any $n\geq 2$ and $\bm{\eta}_n \in \R^{nd}$,
\begin{align}\label{eq:limitzero}
 \lim_{R\to \infty}   \frac{1}{R^{2n-2}}\varphi(\bm{\eta}_{n}/R) =0.
    \end{align}
  Moreover, for any $n\geq 2$ and $1\leq k\leq n-1$, when $\eta_{i}\neq 0$ for $1\leq i \leq k$, we have
\begin{align}\label{phibound2}
\varphi(\bm{\eta}_{n}) \leq \left(1\wedge \frac{1}{\|\eta_1\|^2\dots \|\eta_k\|^2} \right)   \varphi(\bm{\eta}_{k,n}).
    \end{align}
\end{lemma}

\begin{proof}
The proof of \eqref{phibound1} and \eqref{eq:limitzero} is trivial. For \eqref{phibound2},  we have that for any $n\geq 2$,
\begin{align*}
 \varphi(\bm{\eta}_n) &=    \int_{S_n} e^{-w_1 \|\eta_1\|^2} e^{- \sum_{i=2}^n w_i \|\eta_i\|^2}  \dd \bm{w}_{n} \\
 &= \int_{\R_{>0}^{n-1}}\1_{(\sum_{i=2}^nw_i)<1}  \frac{1}{\|\eta_1\|^2} \left(
 1-e^{-\left(1-\sum_{i=2}^nw_i\right)\|\eta_1\|^2}\right)e^{- \sum_{i=2}^n w_i \|\eta_i\|^2} \dd w_2\dots \dd w_n \\
 &\leq \left(1 \wedge\frac{1}{\|\eta_1\|^2}\right) \int_{S_{n-1}}  e^{- \sum_{i=2}^n w_{i-1} \|\eta_i\|^2} \dd \bm{w}_{n-1}=\left(1 \wedge\frac{1}{\|\eta_1\|^2}\right)\varphi(\bm{\eta}_{1,n}).
\end{align*}
{Here in the second line, we integrated $w_1$ from $0$ to $1-\sum_{i=2}^n w_i$. To check the upper bound, we note that using $1-e^{-x} \leq x$,  \[\1_{(\sum_{i=2}^nw_i)<1} \frac{1}{\|\eta_1\|^2} \left(
 1-e^{-\left(1-\sum_{i=2}^nw_i\right)\|\eta_1\|^2}\right) \leq \1_{(\sum_{i=2}^nw_i)<1} \frac{1}{\|\eta_1\|^2}(1-\sum_{i=2}^nw_i)\|\eta_1\|^2 \leq 1. \] }
We can then conclude {\eqref{phibound2}} by induction. 
\end{proof}

We are now ready to prove Proposition~\ref{prop:firstchaos} and Proposition~\ref{prop:higher}.
\subsubsection{Proof of Proposition~\ref{prop:firstchaos}}
For any $R>1$, $\eta_1\in\R^d$, it is easy to see that  for any $\eta_1 \neq 0$, \[
\varphi({\eta_1}/R)=\int_{0}^1 e^{-{w_1 \|\eta_1/R\|^2}}\dd {w_1}=\frac{1-e^{-\|\eta_1/R\|^2}}{\|\eta_1/R\|^2} \leq 1,
\]
and   $\lim_{R\to \infty} \varphi({\eta_1}/R) =1$. Applying  the dominated convergence theorem with the bound \eqref{eq:finiteJint}, we have
\[
\lim_{R\to \infty}\int_{\bbR^d}   \|\eta_1\|^{-2d+2}J_{d/2}(\|\eta_1\|)^2 \varphi(\eta_1/R)\dd\eta_1 = \int_{\bbR^d}   \|\eta\|^{-2d+2}J_{d/2}(\|\eta\|)^2\dd\eta.
\] Using 
Lemma~\ref{le:nthchaos} with $n=1$, it remains to verify that the variance $\sigma^2$ in \eqref{eq:1stvar} equals to
\[
\kappa^2 (2\pi)^d c_d \int_{\R^d} \|\eta\|^{-2d+2}J_{d/2}(\|\eta\|)^2\dd\eta.
\]

{In fact, by purely computation of Fourier transforms, we have
\[ 
(2\pi)^d c_d \int_{\R^d} \|\eta\|^{-2d+2}J_{d/2}(\|\eta\|)^2\dd\eta = \int_{B_1}\int_{B_1} \frac{1}{\|x-y\|^2}\dd x \dd y .
\]
To further evaluate the integral, we note that
\[
\begin{aligned}
\int_{B_1}\int_{B_1} \frac{1}{\|x-y\|^2}\dd x \dd y &= \int_{\R^d}\int_{\R^d} \1_{B_1}(x) \1_{B_1}(x-z) \frac{1}{\|z\|^2}\dd x \dd z = \int_{\R^d} V(z) \frac{1}{\|z\|^2}  \dd z,
\end{aligned}
\]
where $V(z) = \int_{\R^d} \1_{B_1}(x) \1_{B_1}(x-z)  \dd x = |B_1(0)\cap B_1(z)|$ is the volume of overlap of two unit balls centered at 0 and $z$ respectively. By symmetry, $V(z)$ is a function of the distance $\|z\|$.}

{
Let $r:= \|z\|$. Also, let $\omega_{d-1} = \frac{2 \pi^{d/ 2}}{\Gamma(\frac{d}{2})} $ and $\nu_{d-1}=\frac{\pi^{(d-1)/ 2}}{\Gamma(\frac{d+1}{2})}$ be  the surface area and  volume of $(d-1)$-dimension unit sphere respectively. Then we can compute and integrate the hyperspherical caps as
\[
\begin{aligned}
 &\int_{\R^d} V(z) \frac{1}{\|z\|^2}  \dd z = \int_0^2 \left( \int_{\frac{r}{2}}^1 2\nu_{d-1} (1-t^2)^{\frac{d-1}{2}} \dd t \right)\frac{1}{r^2}\omega_{d-1}  r^{d-1}\dd r\\
 &= 2\nu_{d-1} \omega_{d-1} \int_0^1 \left(\int_0^{2t} r^{d-3} \dd r \right) (1-t^2)^{\frac{d-1}{2}} \dd t =  \frac{2\nu_{d-1} \omega_{d-1}}{d-2} \int_0^1 (2t)^{d-2}(1-t^2)^{\frac{d-1}{2}} \dd t \\
 &= \frac{\nu_{d-1} \omega_{d-1} 2^{d-2}}{d-2} \int_0^1 u^{\frac{d-3}{2}}(1-u)^{\frac{d-1}{2}} \dd u = \frac{\nu_{d-1} \omega_{d-1} 2^{d-2}}{d-2}\frac{\Gamma(\frac{d-1}{2})\Gamma(\frac{d+1}{2})}{\Gamma(d)} = \frac{ 2^{\,d-1}\,\pi^{\,d-\frac12}\,\Gamma(\frac{d-1}{2})}
{(d-2)\,\Gamma(\frac d2)\,\Gamma(d)},
\end{aligned} 
\] where  we make a change of variable $u=t^2$ and use the relation between beta function and gamma function in the last line.
}

\subsubsection{Proof of Proposition~\ref{prop:higher}}
By Lemma~\ref{le:nthchaos}, for any $n\geq 2$, $R>1$, setting $\eta_0=0$, we have
\begin{align*}
        & \frac{1}{R^{2d-2}} {\bbVar\left[I_1^{(n)}(\1_{B_R})\right]}\\&=\kappa^{2n}
 (2\pi)^dR^{2-2n}  c_d^n \int_{\bbR^{nd}} \prod_{j=1}^n \|\eta_j-\eta_{j-1}\|^{-d+2} \|\eta_n\|^{-d}J_{d/2}(\|\eta_n\|)^2\varphi(\bm{\eta}_n/R)  \dd \bm{\eta}_n
        \\&= \kappa^{2n} (2\pi)^d c_d^n[\mathcal{J}^{(n)}(R)+\mathcal{K}^{(n)}(R)],
    \end{align*}
where we define the functions
\[
\mathcal{J}^{(n)}(R):=R^{2-2n} \int_{\R^{nd}} \1_{B_1^c}(\eta_1)  \prod_{j=1}^n \|\eta_j-\eta_{j-1}\|^{-d+2} \|\eta_n\|^{-d}J_{d/2}(\|\eta_n\|)^2\varphi(\bm{\eta}_n/R)  \dd \bm{\eta}_n,
\]
and
\[
\mathcal{K}^{(n)}(R):=R^{2-2n} \int_{\R^{nd}} \1_{B_1}(\eta_1)  \prod_{j=1}^n \|\eta_j-\eta_{j-1}\|^{-d+2} \|\eta_n\|^{-d}J_{d/2}(\|\eta_n\|)^2\varphi(\bm{\eta}_n/R)  \dd \bm{\eta}_n.
\]

The proof then follows from the following two lemmas.
\begin{lemma}\label{le:B1part}
     For  any $0<\kappa<\frac{d-2}{2}$, we have
     \[\lim_{R\to\infty}
     \sum_{n=2}^{\infty}
     \kappa^{2n} c_d^n\mathcal{J}^{(n)}(R)=0.
     \]
\end{lemma}
\begin{proof}
By \eqref{phibound2},  we have that for any $R>1$, $n\geq2$, $\bm{\eta}_n\in \R^{nd}$ such that $\eta_i\neq 0$ for $i=1,\dots, n-1$,
\begin{equation}\label{eq:phibdn}
R^{2-2n}\varphi(\bm{\eta}_n/R) \leq R^{2-2n}\prod_{i=1}^{n-1}\|\eta_i/R\|^{-2}= \prod_{i=1}^{n-1}\|\eta_i\|^{-2}.
\end{equation}
It follows that for any $R>1, n\geq2$,
\begin{equation}\label{eq:jnbd}
    \mathcal{J}^{(n)}(R) \leq \int_{\R^{nd}} \1_{B_1^c}(\eta_1)  \prod_{j=1}^n \|\eta_j-\eta_{j-1}\|^{-d+2} \|\eta_n\|^{-d}J_{d/2}(\|\eta_n\|)^2 \prod_{i=1}^{n-1}\|\eta_i\|^{-2} \dd \bm{\eta}_n.
\end{equation}

We first prove that for any $n\geq 2$, there exists some constant $C(n,d)>0$ such that 
the right-hand-side of \eqref{eq:jnbd} is bounded from above by $C(n,d)$.
To show this, we use the fact that for any $x \in B_1^c$ and any $\gamma>0$, 
\[
\|x\|^{-2}\leq \|x\|^{-2+\gamma}.
\]
Thus for some $\gamma>0$ that we will choose later, we have
\begin{align*}
&\int_{\R^{nd}} \1_{B_1^c}(\eta_1)  \prod_{j=1}^n \|\eta_j-\eta_{j-1}\|^{-d+2} \|\eta_n\|^{-d}J_{d/2}(\|\eta_n\|)^2 \prod_{i=1}^{n-1}\|\eta_i\|^{-2}  \dd \bm{\eta}_n\\
    &\leq \int_{\R^{nd}} \|\eta_1\|^{-2+\gamma}  \prod_{j=1}^n \|\eta_j-\eta_{j-1}\|^{-d+2} \|\eta_n\|^{-d}J_{d/2}(\|\eta_n\|)^2 \prod_{i=2}^{n-1}\|\eta_i\|^{-2} \dd \bm{\eta}_n \\
    & = \int_{\R^{nd}} \|\eta_1\|^{-d+\gamma}  \prod_{j=2}^n \left( \|\eta_j-\eta_{j-1}\|^{-d+2} \|\eta_j\|^{-2}\right)\|\eta_n\|^{-d+2}J_{d/2}(\|\eta_n\|)^2 \dd \bm{\eta}_n.
\end{align*}
When $0<\gamma<d-2$,
for any $n\geq 2$,
applying Lemma~\ref{le:rieszcomp} $(n-1)$ times, the above equals to
\begin{equation}\label{eq:intofetan}
    (k_{\gamma,2,d})^{n-1} \int_{\R^{d}}
\|\eta_n\|^{-2d+\gamma+2} J_{d/2}(\|\eta_n\|)^2 \dd {\eta}_n.
\end{equation}
Let $\gamma = \frac{d-2}{2}$. The integral in \eqref{eq:intofetan} is finite by \eqref{eq:finiteJint}, and \[k_{\gamma,2,d} = \pi^{d/2}\frac{\Gamma(\frac{\gamma}{2})\Gamma(\frac{d-\gamma-2}{2})}{\Gamma(\frac{d-\gamma}{2})\Gamma(\frac{d-2}{2})\Gamma(\frac{2+\gamma}{2})} =\pi^{d/2}\frac{1}{\frac{d-\gamma-2}{2}\Gamma(\frac{d-2}{2})\frac{\gamma}{2}}=\frac{1}{c_d(d-2-\gamma)\gamma}=\frac{1}{c_d\gamma^2}.\]
Thus the right-hand-side of equation \eqref{eq:jnbd} is bounded from above by \[C(n,d):=\left(\frac{1}{c_d}\right)^{n-1}\left(\frac{2}{d-2}\right)^{2n-2}\int_{\R^{d}}
\|\eta\|^{-2d+\frac{d+2}{2}} J_{d/2}(\|\eta\|)^2 \dd {\eta}<\infty.\]
{Here the integral term in $C(n,d)$ does not depend on $n$. Thus  
\begin{equation}\label{eq:infinitesum}
\sum_{n=2}^{\infty}
     \kappa^{2n} c_d^n C(n,d)<\infty\end{equation}
if and only if the geometric series converges, which is equivalent to the condition $0<\kappa<\frac{d-2}{2}$.}

Now  for any $n\geq 2$, 
\begin{equation}\label{eq:jasympt} 
    \mathcal{J}^{(n)}(R)\to 0 \quad \text{as }R\to \infty.
\end{equation}
In fact, using again the bound
\eqref{eq:phibdn} and   that the right-hand-side of \eqref{eq:jnbd} is bounded, \eqref{eq:jasympt} follows from the dominated convergence theorem and \eqref{eq:limitzero}.

By \eqref{eq:jnbd}, $\mathcal{J}^{(n)}(R)\leq C(n,d)$ for all $R>1, n\geq 2$. 
The proof is complete once we apply the dominated convergence theorem again to the sum  of $\kappa^{2n}c_d^n\mathcal{J}^{(n)}(R)$ from $n=2$ to infinity  and use the results \eqref{eq:infinitesum} and \eqref{eq:jasympt}.
\end{proof}

\begin{lemma}\label{le:B1cpart}
    For  any $0<\kappa<\frac{d-2}{2}$, we have
     \[\lim_{R\to\infty}
     \sum_{n=2}^{\infty}
     \kappa^{2n} c_d^n\mathcal{K}^{(n)}(R)=0.
     \]
\end{lemma}
\begin{proof}
Similar to the proof of Lemma~\ref{le:B1part}, we first show that, for any $R>1$, $n\geq 2$, $\mathcal{K}^{(n)}(R) \leq C'(n,d)$ for some constant $C'(n,d)>0$ . In order to have $\kappa^{2n}c_d^nC'(n,d)$  be summable from $n=2$ to infinity for all $0<\kappa<\frac{d-2}{2}$, we conduct the following analysis.

Let $R>1, n\geq 2$, and $m$ be an integer so that $0\leq m\leq n-2$. 
By \eqref{phibound2}, 
for any  
$\bm{\eta}_n\in \R^{nd}$
such that
$\eta_i\neq 0$ for $i=m+2,\dots, n-1$, 
we have
\begin{equation}\label{eq:phibdmn}
    R^{2-2n}\varphi(\bm{\eta}_n/R) \leq R^{2-2n}\prod_{i=m+2}^{n-1}\|\eta_i/R\|^{-2}= R^{-2m-2 }\prod_{i=m+2}^{n-1}\|\eta_i\|^{-2}.
\end{equation}
It follows that for any $R>1, n\geq2$ and $0\leq m \leq n-2$,
\begin{equation}\label{eq:kndmn}
    \mathcal{K}^{(n)}(R) \leq R^{-2m-2}\int_{\R^{nd}} \1_{B_1}(\eta_1)  \prod_{j=1}^n \|\eta_j-\eta_{j-1}\|^{-d+2} \|\eta_n\|^{-d}J_{d/2}(\|\eta_n\|)^2 \prod_{i=m+2}^{n-1}\|\eta_i\|^{-2}\dd \bm{\eta}_n.
\end{equation}
Note that for any $x \in B_1$ and any $\gamma'\geq 0$, 
\[
\|x\|^{-d+2}\leq \|x\|^{-d+2-\gamma'}.
\]
Thus for some $\gamma' \geq 0$ and some $0\leq m \leq n-2$ that we will choose later, we have
\begin{align*}
&\int_{\R^{nd}} \1_{B_1}(\eta_1)  \prod_{j=1}^n \|\eta_j-\eta_{j-1}\|^{-d+2} \|\eta_n\|^{-d}J_{d/2}(\|\eta_n\|)^2 \prod_{i=m+2}^{n-1}\|\eta_i\|^{-2}  \dd \bm{\eta}_n\\
    &\leq \int_{\R^{nd}} \|\eta_1\|^{-d+2-\gamma'}  \prod_{j=2}^n \|\eta_j-\eta_{j-1}\|^{-d+2} \|\eta_n\|^{-d}J_{d/2}(\|\eta_n\|)^2 \prod_{i=m+2}^{n-1}\|\eta_i\|^{-2} \dd \bm{\eta}_n.
\end{align*}
When $0<2-\gamma'<d-2(m+1)$, applying Lemma~\ref{le:rieszcomp} $(m+1)$ times,
the above equals to 
\begin{align*}
    \prod_{i=0}^{m}k_{2i+2-\gamma',2,d}&\int_{\R^{(n-m-1)d}}\|\eta_{m+2}\|^{-d+2m+4-\gamma'}\\&\quad\quad
\prod_{j=m+3}^n \|\eta_j-\eta_{j-1}\|^{-d+2} \|\eta_n\|^{-d}J_{d/2}(\|\eta_n\|)^2 \prod_{i=m+2}^{n-1}\|\eta_i\|^{-2} \dd \bm{\eta}_{m+1,n}.
\end{align*}
Then again applying Lemma~\ref{le:rieszcomp} $(n-2-m)$ times, the above further equals to
\begin{equation}\label{eq:kbdmnfinal}
\prod_{i=0}^{m}k_{2i+2-\gamma',2,d}\left(k_{2m+2-\gamma',2,d}\right)^{n-2-m}\int_{\R^{d}}\|\eta_{n}\|^{-2d+2m+4-\gamma'} J_{d/2}(\|\eta_n\|)^2  \dd {\eta}_n.
\end{equation}
Note that when $0<2m+4-\gamma'<d$, the integral in \eqref{eq:kbdmnfinal} is finite by \eqref{eq:finiteJint}, and
\begin{equation*}\label{eq:kmgamma}
\begin{aligned}
    k_{2m+2-\gamma',2,d} &= \pi^{d/2}\frac{\Gamma\left(\frac{2m+2-\gamma'}{2}\right)\Gamma\left(\frac{d-2-(2m+2-\gamma')}{2}\right)}{\Gamma\left(\frac{d-(2m+2-\gamma')}{2}\right)\Gamma\left(\frac{d-2}{2}\right)\Gamma\left(\frac{2+(2m+2-\gamma')}{2}\right)} \\&=
\pi^{d/2}\frac{1}{\left(\frac{d-2-(2m+2-\gamma')}{2}\right)\Gamma\left(\frac{d-2}{2}\right)\left(\frac{2m+2-\gamma'}{2}\right)} \\&=
 \frac{1}{c_d(d-2-(2m+2-\gamma'))(2m+2-\gamma')}.
\end{aligned} 
 \end{equation*}
 We need to choose $m$ and $\gamma'$ in some clever way so that the series $\kappa^{2n}c_d^n\left(k_{2m+2-\gamma',2,d}\right)^{n-2-m}$ is summable.
 The crucial observation here is that for any $d\geq 3$, we can find an integer $m_0\geq 0$ and some $ \gamma'_0\in \{0,1/2,1,3/2\}$ such that \begin{equation}\label{eq:mgammaCon}
2m_0+2-\gamma'_0 = \frac{d-2}{2}.
 \end{equation}
In fact, let $m_0$ be the largest nonnegative integer such that $4m_0\leq d-3$ and let $\gamma'_0=\frac{3}{2}+\frac{4m_0-(d-3)}{2}$. It is easy to see that $0\leq \gamma_0'<2$,  \eqref{eq:mgammaCon} is satisfied, and $0<2m_0+4-\gamma_0'<d$ when $d\geq 3$. The values of $m_0$ and $\gamma'_0$ only depend on the dimension $d$. With $m=m_0, \gamma'=\gamma'_0$, we have
\[k_{2m+2-\gamma',2,d} = \frac{1}{c_d (\frac{d-2}{2})^2},\]
and $\kappa^{2}c_dk_{2m+2-\gamma',2,d}<1$  for all $0<\kappa<\frac{d-2}{2}$. 

Note that $m_0$ only depends on the dimension $d$ and could be greater than some small $n$, so we need to choose $m$ differently according to the values of $n$. For any $2\leq n \leq m_0+2$, we choose $m=0$ and $\gamma'=3/2$; for any $n > m_0+2$, we choose $m=m_0$ and $\gamma'=\gamma'_0$. With these choices,  the  conditions 
$\gamma'\geq 0$, $0\leq m \leq n-2$ and $0<2-\gamma'<d-2(m+1)$ 
 are always satisfied, and $\kappa^{2n}c_d^n\left(k_{2m+2-\gamma',2,d}\right)^{n-2-m}$ is summable from $n=2$ to infinity for all $0<\kappa<\frac{d-2}{2}$. 
 By \eqref{eq:kndmn}, define
 $
 C'(n,d)$ as the value of \eqref{eq:kbdmnfinal} under such choices of $m$ and $\gamma'$,
 we have $\mathcal{K}^{(n)}(R)\leq C'(n,d)$ for any $R>1$ and $n\geq 2$, and also 
 \[\sum_{n=2}^\infty \kappa^{2n}c_d^nC'(n,d)<\infty,\]   for any $0<\kappa<\frac{d-2}{2}$.

It also follows from \eqref{eq:kndmn} that for any $n\geq 2$, $\mathcal{K}^{(n)}(R)\leq R^{-2m-2}C'(n,d)$, so  $\mathcal{K}^{(n)}(R)\to 0$ as $R\to \infty$. The proof is thus complete once we apply the dominated convergence theorem to the sum of $\kappa^{2n}c_d^n\mathcal{K}^{(n)}(R)$ from $n=2$ to infinity and use the above results together.
\end{proof}

%
%

\section*{Acknowledgments}
We thank Yu Gu for introducing us to the paper \cite{mueller2004singular}, encouraging us to study this problem, and for many {helpful} discussions. We also thank Carl Mueller, Guangqu Zheng, and Davar Khoshnevisan for {insightful} conversations. R.T. is partially supported by Yu Gu's NSF grant DMS-2203014. This work started taking its shape during the workshop ``Universality and Integrability in KPZ'' at Columbia University, March 15–19, 2024, which was supported by NSF grant DMS-2400990.   We also thank the referees for helpful suggestions.

\bibliographystyle{alpha}
\footnotesize\bibliography{ref}    


\end{document}